\documentclass[10pt]{amsart}
\usepackage{amscd}

\usepackage{amsmath,amssymb,enumerate, %datetime,
 epsfig,amsthm,sidecap,mathrsfs}
\usepackage{amsfonts} 
\usepackage{fancyhdr}
\usepackage{pb-diagram}
\usepackage{color}
\usepackage{hyperref}
\usepackage[mathscr]{euscript}
\usepackage{mathtools}

\let\oldmarginpar\marginpar
\renewcommand\marginpar[1]{\-\oldmarginpar[\raggedleft\footnotesize #1]%
{\raggedright\footnotesize #1}}

\newtheorem{theorem}{Theorem}[section]

\newtheorem{corollary}[theorem]{Corollary}
\newtheorem{proposition}[theorem]{Proposition}

\theoremstyle{definition}
\newtheorem{definition}[theorem]{Definition}

\theoremstyle{remark}

\numberwithin{equation}{section}

\usepackage{enumitem}

\def\Z{\mathbb Z}
\def\Q{\mathbb Q}

\def\kk{\Bbbk}
\def\N{\mathbb N}

\def\tnu{\widetilde{\nu}}
\def\texp{\exp_\tau}
\def\stau{(S^{-2}\tau)}
\def\dt{d_\tau}

\newcommand{\htau}{S^{-2}\tau}
\newcommand{\eps}{\varepsilon}
\newcommand{\ot}{\otimes}
\newcommand{\tG}{\Gamma^\tau}
\newcommand{\qt}{q_\tau}

\DeclareMathOperator{\End}{End}

\DeclareMathOperator{\tr}{tr}
\DeclareMathOperator{\ev}{ev}
\DeclareMathOperator{\coev}{coev}
\DeclareMathOperator{\Id}{Id}
\DeclareMathOperator{\ord}{ord}
\DeclareMathOperator{\Rep}{Rep}
\DeclareMathOperator{\lcm}{lcm}
\DeclareMathOperator{\GL}{GL}

\title[Twisted exponents and twisted FS indicators for Hopf algebras]{Twisted exponents and twisted Frobenius--Schur indicators for Hopf algebras}
\author{Daniel S. Sage}\address{Department of Mathematics\\
  Louisiana State University\\
  Baton Rouge, LA 70803}\email{sage@math.lsu.edu}
\author{Maria D. Vega}\address{Department of Mathematics\\
  North Carolina State University\\
  Raleigh, NC 27695}\email{mdvega@ncsu.edu}

\thanks{The first author was partially supported by a grant from the Simons
  Foundation (\#281502), and %The research of
  the second author was partially supported by NSF grant
  DMS-0946431.}  \subjclass[2010]{Primary:16T05; Secondary:
  20C15} \keywords{Hopf algebra, exponent, Frobenius--Schur indicator}
\begin{document}

\begin{abstract}
  Classically, the exponent of a group is the least common multiple of
  the orders of its elements.  This notion was generalized by Etingof
  and Gelaki to  Hopf algebras. Kashina, Sommerh\"auser
  and Zhu later observed that there is a strong connection between
  exponents and Frobenius--Schur indicators.  In this paper, we
  introduce the notion of twisted exponents and show  there is a
  similar relationship between the twisted exponent and the twisted
  Frobenius--Schur indicators defined in previous work of the authors.
  In particular, we exhibit a new formula for the twisted indicators and use it to prove periodicity and rationality statements.
 \end{abstract}
\maketitle

\section{Introduction}
Classically, the exponent of a group $G$ is the least common multiple
of the orders of its elements.  More generally, the exponent of a
representation $(V,\rho)$ of $G$ is the exponent of
$\rho(G)\subset\GL(V)$.  In~\cite{EtingofGelaki:1999}, Etingof and
Gelaki, building on work of Kashina~\cite{Kashina:1999}, extended this
notion from groups to Hopf algebras.  Kashina, Sommerh\"auser and Zhu
later observed that there is a strong connection between exponents and
Frobenius--Schur indicators~\cite{KashinaSommerZhu:2006}.  Suppose
that $H$ is a semisimple Hopf algebra over an algebraically closed
field of characteristic zero, and let $V$ be a representation of $H$.
They found a formula for the FS indicators $\nu_m(V)$ which implies
that the indicators are $\exp(V)$-periodic and that the indicators lie
in the cyclotomic ring generated by the $\exp(V)$-th roots of unity.

Suppose that $H$ is a Hopf algebra endowed with a fixed automorphism
$\tau$ of finite order.  In this paper, we introduce a new invariant
called the twisted exponent of $V$ with respect to $\tau$.  For
groups, the twisted exponent has a simple explicit description.
Recall that the norm of an element $g\in G$ (in the sense
of~\cite{BumpGinz:2004}) is the product $g\tau(g)\dots\tau^{r-1}(g)$;
here, $r$ is the order of $\tau$.  The twisted exponent of $G$ is then
the least common multiple of the orders of the norms of the group
elements.  Our primary goal is to show that for $H$ semisimple over an
algebraically closed field of characteristic zero, the results of
\cite{KashinaSommerZhu:2006} generalize to the twisted context.  More
precisely, we exhibit a new formula for the twisted Frobenius--Schur
indicators defined in~\cite{SageVega:2012} and use it to prove
periodicity and rationality statements for the twisted indicators.

% Classically, the exponent of a group is the least common
% multiple of the orders of its elements.  Etingof and Gelaki, in
% ~\cite{EtingofGelaki:1999}, generalized this notion from group theory
% into Hopf algebras.  It was later observed by Kashina, Sommerh\"auser
% and Zhu, in ~\cite{KashinaSommerZhu:2006}, that there is a strong
% connection between exponents and Frobenius--Schur indicators.  In this
% paper, we introduce the notion of twisted exponents and show that
% there is also a relation to twisted Frobenius--Schur indicators as
% defined in ~\cite{SageVega:2012}.
\section{Notation}
%Let $\kk$ be an algebraically closed field of characteristic $0$, and 
Let $H$ be a  Hopf algebra over a field $\kk$ with
comultiplication
$\Delta$, counit $\varepsilon$, and bijective antipode $S$.  
We will use the usual Sweedler notation for
iterated comultiplication: if $h\in H$, then 
\begin{equation*}\Delta^{m-1}(h)=\sum_{(h)}{h_1 \otimes h_2
  \otimes \cdots \otimes h_m}.
\end{equation*}
We will denote iterated multiplication by $\mu^m:H^{\otimes m}\to H$,
i.e., $\mu(l_1 \otimes l_2 \otimes \ldots \otimes l_m)= l_1l_2 \ldots
l_m$.  Let $\Rep(H)$ be the category of finite-dimensional left
$H$-modules; we only consider such $H$-modules in this paper.

% \begin{equation*}\mu^m:= \mu \circ (\mu^{m-1} \otimes \Id):
% H^{\otimes m} \rightarrow H
% \end{equation*}
% where $\mu^2 =\mu$ and $\mu^1=\Id$, i.e., $$\mu(l_1 \otimes l_2
% \otimes \ldots \otimes l_m)= l_1l_2 \ldots l_m,$$ where $l_1 \otimes
% l_2 \otimes \cdots \otimes l_m \in H^{\otimes m}$.

Let $\tau$ be a Hopf algebra automorphism of $H$; in particular, $\tau$ commutes with $S$.  We will always assume that $\tau$ has
finite order $r$.  For all $j \in \N$, we define the \emph{$j$-th
  twisted Sweedler power} of $h \in H$ to be
 \begin{equation*}\tilde{h}_\tau^{[j]}:= \mu^{j} \circ \left(\Id \otimes \tau \otimes \cdots \otimes \tau^{j-1}\right)\circ \Delta^{j-1}(h).
 \end{equation*}
 More explicitly,  we have \begin{equation*}\tilde{h}_\tau^{[j]}=\sum_{(h)}{\left(h_1\tau\left(h_2\right)\cdots \tau^{j-1}\left(h_j\right)\right)}.
 \end{equation*}  We will write $\tilde{h}^{[j]}$ for
 $\tilde{h}_\tau^{[j]}$ when the automorphism is clear from context.

% ??????????????
%  Let $\Rep(H)$ be the category of finite-dimensional left $H$-modules;
%  we only consider finite-dimensional $H$-modules.  Throughout the
%  paper, an automorphism of $H$ will always refer to a Hopf algebra
%  automorphism (or equivalently, a bialgebra automorphism).  In
%  particular, such an automorphism commutes with the antipode.  Let
%  $\tau: H \rightarrow H$ be an automorphism of order $r$, such that
%  $\tau^m=\Id$ for some $m \in \N$. Define $k \in \N$ such that $m=kr$.
%  Let $(V,\rho)$ be an $H$-module with corresponding character $\chi$.
%  For all $j \in \N$, we define the \emph{twisted $j$-th Sweedler
%    power} of $h \in H$ to be
%  $$\tilde{h}^{[j]}:= \mu^{j} \circ \left(\Id \otimes \tau \otimes
%    \cdots \otimes \tau^{j-1}\right)\circ \Delta^{j-1}(h).$$  More
%  explicitly,  we have
%  $$\tilde{h}^{[j]}=\sum_{(h)}{\left(h_1\tau\left(h_2\right)\cdots
%      \tau^{j-1}\left(h_j\right)\right)}.$$  Note that if $\tau=\Id$,
%  then, in the notation of \cite{LinMont:2000},
%  $\tilde{h}^{[j]}=h^{[j]}$.
% , matching up with the notation used by Linchenko and Montgomery in \cite{LinMont:2000}. 
%  If $\tau=\Id$, then $\tilde{h}^{[j]}=h^{[j]}$, matching up with the notation used by Linchenko and Montgomery in \cite{LinMont:2000}. 
 %and Kashina, Sommer\"auser and Zhu in \cite{KashinaSommerZhu:2006}.
%??????????????? 

Suppose that $A$ is a linear endomorphism of $H$.  To simply notation, we
write $A^{\otimes [k,k+m)}:=A^k\otimes A^{k+1}\otimes\dots\otimes
A^{k+m-1}\in\End(H^{\otimes m})$.  In particular, $A^{\otimes
  [k,k+m)}=(A^k)^{\otimes m}\circ A^{\otimes [0,m)}$.   As an example,
  note that 
  $\tilde{h}^{[j]}= \mu^{j} \circ \tau^{\otimes [0,j)}\circ
  \Delta^{j-1}(h)$.

\section{Twisted exponents}
%%%%%%%%%%%%%%%%%%%%%
\subsection{Definition and first examples}\label{sec:def}
%Obtain a twisted version of KSZ's formula about twisted FS indicators.

\begin{definition}\label{def:texpH}
 The \emph{twisted exponent}
$\texp(H)$ of $H$ is the smallest $k\in \N$ such that 
\begin{equation}\label{eq:texpH}
\mu^{kr} \circ \left(\Id \otimes \stau \otimes \cdots \otimes \stau^{(kr-1)}\right)\circ \Delta^{kr-1}= \varepsilon \cdot \Id
\end{equation}
or $\infty$ if no such $k$ exists.  When $\texp(H)$ is finite, we set
$\dt=r\texp(H)$.  We will denote the endomorphism on the left side of
\eqref{eq:texpH} by $\tG_{kr}(H)$ or simply $\tG_{kr}$ when $H$ is clear
from context.  More generally, the twisted exponent $\texp(V)$ of
an $H$-module $(V,\rho)$ is the smallest $k$ such that
\begin{equation}\label{eq:texpV}\rho(\mu^{kr} \circ (\Id \otimes (S^{-2}\tau) \otimes
  \cdots \otimes (S^{-2}\tau)^{kr-1}) \circ\Delta^{kr-1}(h))=\varepsilon(h)
  \cdot 1_V
\end{equation}
for all $h\in H$. 
\end{definition}
It is obvious that the twisted exponent of the regular representation
of $H$ coincides with $\texp(H)$.

% More compactly,

% $$\mu^{kr} \circ \left(\Id \otimes \stau \otimes \cdots \otimes \stau^{(kr-1)}\right)\circ \Delta^{kr-1}= \varepsilon \cdot 1.$$
% More generally, the twisted exponent $\texp(V)$ of an $H$-module $(V,\rho)$ is
% the smallest $k$ such that \begin{equation*}\rho(\mu^{kr} \circ (\Id \otimes (S^{-2}\tau) \otimes \cdots \otimes (S^{-2}\tau)^{kr-1})
% \Delta^{kr-1}(h))=\varepsilon(h) \cdot 1_V
% \end{equation*}
% for all $h\in H$.  

If $S^2=\Id$, then the defining formula for the twisted exponent
reduces to $\tilde{h}^{[kr]}\cdot v=\varepsilon(h)v$ for all $h\in
H$ and $v \in V$.  In particular, this is the case when $H$ is
finite-dimensional, semisimple, and cosemisimple.

As a first example, we consider twisted exponents for group algebras.
If $G(H)$ is the group of group-like elements in $H$, then the
automorphism $\tau$ induces a norm map $N:G(H)\to G(H)$ given by
$N(g)=g\tau(g)\dots\tau^{r-1}(g)$.

\begin{proposition} \mbox{}\begin{enumerate}\item If $g\in G(H)$, then
    $\ord(N(g))$ divides $\texp(H)$.
\item If $G$ is a group, then  $\texp{\kk[G]}$ is the least common multiple
  of $\ord(N(g))$ for $g\in G$.
\end{enumerate}
\end{proposition}
(Throughout the paper, we make the convention that any positive
integer divides $\infty$.)
\begin{proof} Since $g \in G(H)$, $\Delta(g)=g \otimes g$ and $S^2(g)=g$.
This implies that $S^{-2}\tau (g)=\tau(g)$, so that $(S^{-2}\tau)^{k}
(g)=\tau^{k}(g)$ for all $k \in \N$.  We now compute:
\begin{align*} 
N(g)^{\texp(H)}&=\mu^{\dt} ( g \otimes \tau(g) \otimes \cdots \otimes
\tau^{\dt-1}(g))\\
&=\mu^{\dt} \circ \left(\Id \otimes \stau \otimes \cdots \otimes \stau^{(\dt-1)}\right)\circ \Delta^{\dt-1}(g)=\varepsilon(g) \cdot 1=1.
\end{align*}
This implies that $\ord(N(g))$ divides $ \texp(H)$, as desired.

In the group algebra case, one needs only to observe that $\texp{\kk[G]}$
is the smallest $k$ such that $1=\tilde{g}^{[kr]}=N(g)^k$.
\end{proof}

The same proof shows that if $(V,\rho)$ is a representation of $G$,
then $\texp(V)$ is the least common multiple of $\ord(\rho(N(g)))$.

If $\tau$ is an involution of $G$ and $\texp(\kk[G])=1$, then it is
immediate that $\tau(g)=g^{-1}$ and $G$ is commutative.  A similar
result holds for Hopf algebras.

\begin{proposition}If $\dt\le 2$, then $H$ is commutative and
  cocommutative.  Moreover, $\tau$ is either the identity or the antipode.
\end{proposition}
\begin{proof}
If $\dt=1$ or $\dt=2$ and $r=1$, then $\tau=\Id$, and the
conclusion follows from \cite[Proposition 2.2(6)]{EtingofGelaki:1999}.
 Now suppose that $\dt=2$ and $r=2$, so $\texp(H)=1$.  It follows that
 $h_1 (S^{-2}\tau ) h_2 = \varepsilon (h)1$ for all $h\in H$, so by
 definition of the antipode, $S^{-2}\tau=S$, i.e., $\tau=S^3$.  This
 implies that $\tau$ is an algebra and coalgebra antiautomorphism as
 well as a bialgebra automorphism.  Hence,
 $\tau(x)\tau(y)=\tau(xy)=\tau(y)\tau(x)$ for all $x,y\in H$.  Since
 $\tau$ is a bijection, $H$ is commutative.  Also,
 $\tau(h_1)\otimes\tau(h_2)=\Delta\tau(h)= \tau(h_2)\otimes\tau(h_1)$
 for any $h\in H$, and one obtains the cocommutativity of $H$ by
 applying $\tau^{-1}\otimes\tau^{-1}$ to this equation.  Note that in
 this case, since $S^2=\Id$, $\tau=S$.
\end{proof}

%%%%%%%%%%%%%%%%%%%
%%%%%%%%%%%%%%%%%%%
\subsection{Properties of the twisted exponent}

In this section, we derive twisted analogues of the basic properties
of the exponent.  

We first examine the relationship between $\texp(H)$ and $\texp(V)$
for $(V,\rho)\in\Rep(H)$.
\begin{proposition} \label{divide} Suppose that
    $\texp(V)$ is finite.  Then $k\in\N$ satisfies \eqref{eq:texpV} if
    and only if  $\texp(V)$ divides $k$.
  \end{proposition}
\begin{proof}  Set $e=\texp(V)$.   Note that for any $\ell\ge 0$, 
\begin{equation*}\begin{aligned} 
\rho(\mu^{er}\circ(\htau)^{\otimes[\ell,er+\ell)}&\circ\Delta^{er-1}(h))\\&=\rho(\mu^{er}\circ(\htau)^{\otimes[0,er)}((\htau)^\ell(h_1)\ot\dots\ot(\htau)^\ell(h_{er}))\\
&=\rho(\mu^{er}\circ(\htau)^{\otimes[0,er)}((\htau)^\ell(h)_1\ot\dots\ot(\htau)^\ell(h)_{er})\\
&=\eps((\htau)^\ell(h))1_V=\eps(h)1_V.
\end{aligned}
\end{equation*}
Now, suppose that $k$ satisfies \eqref{eq:texpV}, and write $k=et+\ell$
with $0\le \ell<e$.  Applying the previous observation repeatedly, we obtain
\begin{align*}
%\begin{split} 
&\eps((\htau)^{ter}(h))1_V=\eps(h)1_V=\rho(\mu^{kr}\circ(\htau)^{\otimes[0,kr)}\circ\Delta^{kr-1}(h))\\&=\rho(\mu^{kr}\circ((\htau)^{\ot[0,er)}\ot \ldots \ot (\htau)^{\ot[(t-1)er,ter)}\ot(\htau)^{\ot[ter,(te+\ell)r)})\circ\Delta^{kr-1}(h))\\
&\begin{multlined}=\left(\prod_{i=1}^t
\rho(\mu^{er}\circ(\htau)^{\otimes[(i-1)er,ier)}(h_{(i-1)er+1}\ot\dots\ot
  h_{ier})\right)\\
\rho(\mu^{\ell r}\circ(\htau)^{\otimes[ter,(te+\ell)r)}(h_{ter+1}\ot\dots\ot
  h_{kr}))
\end{multlined}
\\
&=\left(\prod_{i=1}^t \eps(h_i)\right)\rho(\mu^{\ell
  r}\circ(\htau)^{\otimes[ter,(te+\ell)r))}(h_{t+1}\ot\dots\ot
  h_{t+\ell r}))\\
&=\rho(\mu^{\ell r}\circ(\htau)^{\otimes[0,\ell
  r)}\circ(\htau)^{\otimes ter}(\Delta^{\ell r-1}(h)))\\
&=\rho(\mu^{\ell r}\circ(\htau)^{\otimes[0,\ell
  r)}(\Delta^{\ell r-1}((\htau)^{ter}(h))).
%\end{split}
\end{align*}
It follows that $\ell<e$ satisfies \eqref{eq:texpV}, so $\ell=0$.

The same computation with $k=et$ gives the reverse implication.
\end{proof}

\begin{corollary}\label{lcmV} The twisted exponent $\texp(H)$ is the least common multiple of
  $\texp(V)$ for $V\in\Rep(H)$.
\end{corollary}
\begin{proof}  Assume that $\texp(H)$ is finite, as otherwise there
  is nothing to prove. 
%It is obvious that $\texp(H)$ satisfies
% \eqref{eq:texpV} for any $V$, so the Proposition implies that
  %$\texp(H)=\lcm\{\texp(V)\mid V\in\Rep(H)\}\mid \texp(H)$.
It is trivial that $\texp(H)$ divides $\lcm\{\texp(V)\mid
V\in\Rep(H)\}$.
%, since $\texp(H)$ is the twisted exponent of the regular representation
%of $H$. 
Conversely, since $\texp(H)$ satisfies
 \eqref{eq:texpV}, the proposition implies that
  $\texp(V)|\texp(H)$ for each $V$, i.e. $ \lcm\{\texp(V)\mid V\in\Rep(H)\}$ divides  $\texp(H)$. Thus, 
  $\texp(H)= \lcm\{\texp(V)\mid V\in\Rep(H)\}$.
% It is obvious that $\texp(H)$ satisfies
%  \eqref{eq:texpV}, so the Proposition implies that
%   $\texp(V)|\texp(H)$ for each $V$, i.e. $ \lcm\{\texp(V)\mid V\in\Rep(H)\}$ divides  $\texp(H)$. Thus, 
%   $\texp(H)= \lcm\{\texp(V)\mid V\in\Rep(H)\}$.
\end{proof}

We next consider the behavior of the twisted exponent under various
algebraic operations.
%collect other properties of $\texp(H)$ in the following proposition.

\begin{proposition} \mbox{}
\begin{enumerate}

\item Let $A \subseteq H$ be a Hopf subalgebra such that
  $\tau(A)=A$,
  and suppose that $r'=\ord(\tau|_A)$ (so $r'\mid r$).  Then, $\texp(A)$ divides
  $\texp(H)\frac{r}{r'}$.  In particular, if $r=r'$, then
  $\texp(A)\mid\texp(H)$.
\item \label{quotient} Let $\phi:H\to H'$ be a surjective map of Hopf algebras with
  $\tau(\ker(\phi))=\ker(\phi)$, and let $\tau'$ be the Hopf algebra
  automorphism of $H'$ defined by $\tau'(\phi(h))=\phi(\tau(h))$.
  Then, $\exp_{\tau'}(H')$ divides $\texp(H)\frac{r}{r'}$.
\item $K$ is a field extension of $\kk$, then $\texp(H \otimes_{\kk}
  K)= \texp(H)$.
\item If $H$ is finite-dimensional, then $\exp_{\tau^*}(H^*)=\texp(H)$.
\end{enumerate}
\end{proposition}
\begin{proof}
  Set $e=\texp(H)$.  To prove the first statement, it suffices by
  Proposition~\ref{divide} to show that $\tG_{er}(A)=\eps \Id_A$.
  This is clear, since $\tG_{er}(A)=\tG_{er}(H)|_A$.  The proof of
  \eqref{quotient} is similar, using the fact that
  $\Gamma^{\tau'}_{er}(H')\circ
  \phi=\phi\circ\tG_{er}(H)=\phi\circ(\eps \Id_H)=\eps' \Id_{H'}$.
  The fact that the twisted exponent is preserved by extension of
  scalars in obvious.  Finally, if $H$ is finite-dimensional,
  $\tG_{kr}(H)^*=\Gamma^{\tau^*}_{kr}(H^*)$, which implies that
  $\exp_{\tau^*}(H^*)=\texp(H)$.
\end{proof}

Let $H'$ be another Hopf algebra with bijective antipode endowed with
an automorphism $\tau'$ of finite order $r'$.  Then
$\gamma=\tau\ot\tau'$ is an automorphism of $H\ot H'$ with order
$\lcm(r,r')$, and $\exp_{\gamma}(H \otimes H')$ is defined.
\begin{proposition} With $\gamma$ as above, 
  \begin{equation*}\exp_{\gamma}(H \otimes H')=\frac{\lcm(\dt, d_{\tau'})}{\lcm(r,
    r')}.
\end{equation*}
In particular, if $\tau$ and $\tau'$ have the same order, then
%$\exp_{\gamma}(H \otimes H')=\lcm(\dt, d_{\tau'})$.
$\exp_{\gamma}(H \otimes H')=\lcm(\exp_{\tau}(H), \exp_{\tau'}(H'))$.
\end{proposition}
\begin{proof} By definition of the tensor product, %that
  $\Gamma^\gamma_t=\tG_t\ot\Gamma^{\tau'}_t$ for all $t\in\N$.
  Moreover, $\Gamma^\gamma_{k\ord(\gamma)}=\eps_{H\ot H'} \Id_{H\ot
    H'}$ if and only if $\tG_{k\ord(\gamma)}=\eps \Id_H$ and
  $\Gamma^{\tau'}_{k\ord(\gamma)}=\eps' \Id_{H'}$, so this condition
  holds if and only if $\dt$ and $d_{\tau'}$ both divide
  $k\ord(\gamma)$.  Hence, $d_\gamma=\lcm(\dt,d_{\tau'})$, and
  $\exp_{\gamma}(H \otimes H')$ is obtained by dividing this by
  $\ord(\gamma)$.

\end{proof}

%%%%%%%%%%%%%%%%
%%%%%%%%%%%%%%%%%
%%%%%%%%%%%%%%%%%
\subsection{The case of $H$ finite-dimensional and involutory}\label{sec:inv}
%\subsection{A formula for the twisted exponent for $H$ involutory and finite-dimensional}
Assume that $H$ is finite-dimensional and involutory. It is proved
in~\cite{KashinaSommerZhu:2006} that the exponent of any module $V$ is
given by the order of the action on $V\ot H^*$ of a certain element of
$H\ot H^*$.  We show that the twisted exponent can be computed in a
similar way.

Let $\coev:\kk\to H\ot H^*$ be the coevaluation map.  Recall that if
$b_1, \dots, b_n $ is a basis of $H$ with dual basis $b_1^*, \dots,
b_n^*$, then $\coev(1)=\sum b_i\ot b^*_i$.  We now define $\qt\in H\ot
H^*$ via $\qt=\prod_{i=0}^{r-1} (\Id_H\ot (\tau^{*})^{i})\coev(1)$.  More explicitly,  
	\begin{equation}\label{qt}
	\qt = \sum{b_{i_1}\cdots b_{i_r}\otimes b^*_{i_1}\tau^*(b^*_{i_2})\cdots (\tau^*)^{(r-1)}(b^*_{i_r})}. 
	\end{equation} 
        We denote by $\qt|_{V \otimes H^*}$ the action of $\qt$ on
        $V\ot H^*$ by left multiplication.

% 	\begin{remark}
% 	The tensor $q$ is independent of the choice of basis.
% 	\end{remark}

		\begin{proposition}\label{ordofq}  The order of
                  $\qt|_{V \otimes H^*}$ equals $\texp(V)$.
                  % $$\texp(V)= \ord(q|_{V \otimes H^*}).$$
		\end{proposition}
\begin{proof}
Given $k\in\N$, let $m=kr$, so that 
\begin{equation*}
\qt^k= \sum{b_{i_1}\cdots b_{i_m}\otimes b^*_{i_1}\tau^*(b^*_{i_2})\cdots (\tau^*)^{(m-1)}(b^*_{i_m})}.
\end{equation*}
Note that $f\in\End(V\ot H^*)$ is uniquely determined by the maps
$(\Id_V\ot\ev_h)\circ f$.  Thus, the order of $\qt$ is the smallest
$k$ such that $(\Id_V\ot\ev_h)(\qt^k\cdot (v\ot\alpha))=v\ot\alpha$ for
all $v\in V$, $h\in H$, and $\alpha\in H^*$.  Since left
multiplication by $\qt^k$ commutes with the right $H^*$-action, it
suffices to consider $\alpha=\eps$.  Using basic properties of dual
bases, we obtain
\begin{equation*}\begin{split}
(\Id_V &\otimes \ev_h)(\qt^k\cdot(v\ot\eps))\\
 & = \sum b^*_{i_1}(h_1)((\tau^*)(b^*_{i_2}))(h_2)\cdots ((\tau^*)^{(m-1)}(b^*_{i_m}))(h_m)\eps(h_{m+1})(b_{i_1}\cdots
  b_{i_m}\cdot v) \\
&=\sum b^*_{i_1}(h_1)b^*_{i_2}\left(\tau(h_2)\right)\cdots
b^*_{i_m}(\tau^{(m-1)}(h_m)) b_{i_1}\cdots b_{i_m}\cdot v  \\
&=\sum h_1\tau(h_2)\cdots \tau^{(m-1)}(h_m) \cdot
v=\tilde{h}^{[kr]}\cdot v.
\end{split}
\end{equation*}
Since $(\Id_V\ot\ev_h)(\Id_{V\ot H^*}(v\ot\eps))=\eps(h)v$, we conclude
that the order of $\qt|_{V \otimes H^*}$ is the minimum $k$ such that
$\tilde{h}^{[kr]}\cdot v=\eps(h)v$ for all $h$ and $v$.  However, when
$H$ is involutory, this is the definition of $\texp(V)$.

\end{proof}		
\section{Twisted FS indicators and twisted
  exponents}

In this section, we assume that $\kk$ is an algebraically closed field
of characteristic zero and that $H$ is a semisimple Hopf algebra.  In
particular, $H$ is finite-dimensional, so $S$ is bijective; moreover,
a theorem of Radford states that any Hopf algebra automorphism has
finite order~\cite{Radford:1990}.   We denote by $\Lambda$ the unique
(two-sided) integral of $H$ such that $\eps(\Lambda)=1$.

Let $V$ be a representation of $H$ with character $\chi$.  Then, for
any $m\in\N$ divisible by $r$, the $m$th twisted Frobenius-Schur
indicator was defined in \cite{SageVega:2012} to be the character sum
\begin{equation*}
%\label{eq:tdmFS}
 \nu_m(\chi,\tau)=\chi\left(\tilde{\Lambda}^{[m]}\right).
 \end{equation*}
We usually will write $\tnu_m(\chi)$ instead of $\nu_m(\chi,\tau)$
when the automorphism is clear from context.
% We note that this is only defined for $m$ divisible by the order of
% $\tau$.  We will write $\tnu_m(\chi)$ instead of $\nu_m(\chi,\tau)$ when
% this does not cause confusion.
This generalizes both the FS indicators defined for semisimple Hopf
algebras by Linchenko and Montgomery~\cite{LinMont:2000} and the
twisted FS indicators for groups introduced by Bump and
Ginzburg~\cite{BumpGinz:2004}.

The $m$th twisted FS indicator can be computed as the trace of a
certain order $m$ operator~\cite{SageVega:2012}.  This implies that
$\tnu_m(\chi)\in\Z[\zeta_m]$, where $\zeta_m$ is a primitive $m$th
root of unity.  A priori, this allows the possibility that the field
extension of $\Q$ generated by all the $\tnu_m(\chi)$'s is infinite.
However, this is not the case when the twisted exponent is
finite.  Indeed, we now exhibit another trace formula for
$\tnu_m(\chi)$, which can be used to show that the function $m\mapsto
\tnu_m(\chi)$ is $r\texp(V)$-periodic with image lying in
$\Z[\zeta_{\texp(V)}]$.

% It is well-known that if $H$ is semisimple, then 
% $H$ contains a unique two-sided integral, which will be denoted by $\Lambda$, normalized so that
% $\varepsilon(\Lambda)=1$.  
% In \cite{SageVega:2012} we  defined twisted Frobenius--Schur  indicators for  $(V,\rho)$ (or
% $\chi$)for a semisimple Hopf algebras $H$ to be the character sum

% \begin{equation}
% \label{eq:tdmFS}
%  \nu_m(\chi,\tau)=\chi\left(\tilde{\Lambda}^{[m]}\right).
%  \end{equation}

% We note that this is only defined for $m$ divisible by the order of
% $\tau$.  We will write $\tnu_m(\chi)$ instead of $\nu_m(\chi,\tau)$ when
% this does not cause confusion.

% If $\tau=\Id$, this formula coincides with the definition of Linchenko
% and Montgomery~\cite{LinMont:2000}.  Moreover, suppose $H=k[G]$ for a
% finite group $G$.  In this case, $\Lambda=
% \dfrac{1}{|G|}\sum_{g \in G}{g}$, and we recover Bump and Ginzburg's twisted
% Frobenius--Schur indicators for groups~\cite{BumpGinz:2004}.
% \begin{theorem}[twisted version of the second formula {\color{red}FIX LABEL}]\label{2ndformula}  Let $\tau$ be an automorphism of order $r$ such that $\tau^m=\Id$.  Suppose $m=kr$, for some $k \in \N$.  Then, for any $V \in \Rep(H)$ with character $\chi$, the $m$-th twisted Frobenius-Schur indicator is
% 		$$\tnu_m(\chi)=\dfrac{1}{\dim H}\tr\left(q^k|_{\left(V\otimes H^*\right)}\right). $$	
% 		\end{theorem} 

Let $\qt$ be the element of $H\ot H^*$ defined in \eqref{qt}.

\begin{theorem}\label{2ndformula}  If $m=kr$, then the $m$th twisted
Frobenius-Schur indicator is given by
\begin{equation*}\tnu_m(\chi)=\dfrac{1}{\dim H}\tr\left(\qt^k|_{V\otimes H^*}\right).
\end{equation*}
\end{theorem}

\begin{proof} Since the trace of left multiplication by $\alpha\in
  H^*$ on $H^*$ equals $\dim(H)\alpha(\Lambda)$~\cite[Proposition
  2.4]{LarsonRadford:1988}, the trace of left multiplication by
  $h\ot\alpha$ on $V\ot H^*$ is $\dim(H)\chi(h)\alpha(\Lambda)$.  We
  now apply this fact to compute the trace of $\qt^k|_{V\otimes H^*}$:
% \begin{equation*}\begin{split}
% 	\tr\left(\qt^k|_{V\otimes H^*}\right)&= \dim(H) \sum_{i_1,
%           i_2, \ldots, i_m=1}^{\dim H}\chi(b_{i_1}b_{i_2}\ldots b_{i_m})(b^*_{i_1}((\tau^*)^{-1}(b^*_{i_2}))\cdots ((\tau^*)^{-(m-1)}(b^*_{i_m})))(\Lambda)\\
% 	&=\dim(H) \sum\chi(b_{i_1}b_{i_2}\ldots b_{i_m})b^*_{i_1}(\Lambda_1)b^*_{i_2}(\tau(\Lambda_2))\cdots b^*_{i_m}(\tau^{(m-1)}(\Lambda_m))\\
% 	&= \dim(H)\sum\chi(\Lambda_1\tau(\Lambda_2)\ldots \tau^{m-1}(\Lambda_m))\\
% 	&=\dim(H)\chi(\tilde{\Lambda}^{[m]}).
%       \end{split}
%     \end{equation*}
\begin{equation*}\begin{split}
	\tr\left(\qt^k|_{V\otimes H^*}\right)&= \dim(H) \sum\chi(b_{i_1}b_{i_2}\ldots b_{i_m})(b^*_{i_1}((\tau^*)(b^*_{i_2}))\cdots ((\tau^*)^{(m-1)}(b^*_{i_m})))(\Lambda)\\
	&=\dim(H) \sum\chi(b_{i_1}b_{i_2}\ldots b_{i_m})b^*_{i_1}(\Lambda_1)b^*_{i_2}(\tau(\Lambda_2))\cdots b^*_{i_m}(\tau^{(m-1)}(\Lambda_m))\\
	&= \dim(H)\sum\chi(\Lambda_1\tau(\Lambda_2)\ldots \tau^{m-1}(\Lambda_m))\\
	&=\dim(H)\chi(\tilde{\Lambda}^{[m]}).
      \end{split}
    \end{equation*}
    Therefore, \begin{equation*}\tnu_m(\chi)=\chi(\tilde{\Lambda}^{[m]})=\frac{1}{\dim H}\tr\left(\qt^k|_{V\otimes H^*}\right).
    \end{equation*}
   %  since $H$ is finite dimensional. Thus, $$\tnu_m(\chi)=\dfrac{1}{\dim H}\tr\left(q^k|_{\left(V\otimes H^*\right)}\right), $$	as desired.
		
	% 	Recall that  $q$ acts on $V \otimes H^*$
% 	such that the $H$-tensorand  acts by the module action and on the $H^*$- tensorand by left multiplication.  

% 	\begin{eqnarray*}
% 	\tr\left(q^k|_{\left(V\otimes H^*\right)}\right)&=& \dim(H) \sum_{i_1, i_2, \ldots, i_m=1}^{n}{\chi(b_{i_1}b_{i_2}\ldots b_{i_m})(b^*_{i_1}\tau^{-1}(b^*_{i_2})\cdots \tau^{-(m-1)}(b^*_{i_m}))(\Lambda)}\\
% 	&=&\dim(H) \sum{\chi(b_{i_1}b_{i_2}\ldots b_{i_m})b^*_{i_1}(\Lambda_1)(b^*_{i_2}\tau(\Lambda_2))\cdots (b^*_{i_m}\tau^{(m-1)}(\Lambda_m))}\\
% 	&=& \dim(H)\sum{\chi(\Lambda_1\tau(\Lambda_2)\ldots \tau^{m-1}(\Lambda_m))}\\
% 	&=& \dim(H)\sum{\chi(\tilde{\Lambda}^{[m]})}.
% 	\end{eqnarray*}
% Therefore, $$\sum{\chi(\tilde{\Lambda}^{[m]})}=\dfrac{1}{\dim H}\tr\left(q^k|_{\left(V\otimes H^*\right)}\right),$$	 since $H$ is finite dimensional. Thus, $$\tnu_m(\chi)=\dfrac{1}{\dim H}\tr\left(q^k|_{\left(V\otimes H^*\right)}\right), $$	as desired.	
\end{proof}	
When $\tau=\Id$, this result is due to \cite{KashinaSommerZhu:2006}.
%{\color{red} NOTE: does it make sense to add this corollary???}
% \begin{corollary}
%   The twisted FS indicators $\tnu_m(\chi)$ are contained in
%   $\Z[\zeta_{\texp(V)}]$.  In particular, for any $m$ and $\chi$,  $\tnu_m(\chi)\in\Z[\zeta_{\texp(H)}]$.
%   % the\marginpar{Change to integers!}
% %   cyclotomic field $\Q(\texp(V))$.  In particular, they are in
% %   $\Q(\texp(H))$.
% \marginpar{I think it is in this field as opposed to $\Q(\dt)$.}
% \end{corollary}
\begin{corollary} \mbox{}\begin{enumerate} \item If $\texp(V)$ is finite,
    then the function $m\mapsto \tnu_m(\chi_V)$ is
  $r\texp(V)$-periodic with image lying in $\Z[\zeta_{\texp(V)}]$. 
\item If $\texp(H)$ is finite, then this function is
  $r\texp(H)$-periodic for any $V\in\Rep(H)$ and
  $\tnu_m(\chi_V)\in\Z[\zeta_{\texp(H)}]$ for all $m\in r\N$.
\end{enumerate}
% the\marginpar{Change to integers!}
%   cyclotomic field $\Q(\texp(V))$.  In particular, they are in
%   $\Q(\texp(H))$.
\end{corollary}

\begin{proof} Since $H$ is semisimple, it is involutory.  Accordingly,
  by Proposition~\ref{ordofq}, $e=\texp(V)=\ord(\qt|_{V\ot H^*})$,
  whence the periodicity of the indicators.  Moreover, this implies
  that the eigenvalues of $\qt^k|_{V\ot H^*}$ are $e$th roots of unity
  for each $k$.  Since $\tnu_m(\chi)$ is an algebraic
  integer~\cite[Corollary 3.6]{SageVega:2012}, the first statement now
  follows immediately from the theorem while the second is a
  consequence of Corollary~\ref{lcmV}.
% The result follows Theorem \ref{2ndformula} and the fact that the
% eigenvalues of $q^k$ on $V \otimes H^*$ are $\dt$ roots of unity.
\end{proof}

We remark that it is not true that the field $\Q(\tnu_m(\chi_V)\mid m\in
r\N, V\in\Rep(H))$ coincides with $\Q(\zeta_{\texp(H)})$.   For example, if $H$ is
the group algebra of a finite group and $\tau=\Id$, then the FS
indicators are all integers.

%%%%%%%%%
%%%%%%%%%
% \subsection{When is $\tnu_m(\chi)$ an integer?}\marginpar{Should we
%   split into subsections?}

%\marginpar{Removed the split into subsections}

Let $\Q_d$ denote the cyclotomic field $\Q(\zeta_d)$.  The previous
corollary implies that $\tnu_m(\chi)\in
\Q_m\cap\Q_{\texp(V)}=\Q_{\gcd(m,\texp(V))}$.  Accordingly,
  $\tnu_m(\chi)\in\Z$ when $m$ and
  $\texp(V)$ are relatively prime.  However, more can be said.  

Given $m,d\in\N$, we follow the terminology of
\cite{KashinaSommerZhu:2006} and say that $m$ is \emph{large compared
  to $d$} if $d/\gcd(d,m)$ and $m$ are relatively prime.
Equivalently, for any $p$ dividing $m$, the $p$-adic valuation of $m$
is at least as great as the $p$-adic valuation of $d$.

\begin{theorem} Let $\dt(V)=r\texp(V)$. 
\begin{enumerate}	
\item If $m\in r\N$ is large compared to $\dt(V)$, then $\tnu_m(\chi) \in \Z$.
\item If $\dt(V)$ is square-free, then $\tnu_m(\chi) \in \Z$ for all
  $m\in r\N$.  In particular, this is the case when $\dt=\dt(H)$ is square-free.
			\end{enumerate}
                      \end{theorem}
                      
\begin{proof}   Suppose that $m=kr$ is large compared to $\dt(V)$.  Let
  $e=\texp(V)$, and set $e'= \gcd(k,e)$.  By Theorem~\ref{2ndformula},
  \begin{equation*}\tnu_m(\chi)=\dfrac{1}{\dim
      H}\tr\left((\qt^{e'}|_{V\otimes H^*})^{k/e'}\right), 
\end{equation*}
and since $\qt^{e'}|_{V\otimes H^*}$ has order $e/e'$, we see that
$\tnu_m$ is an algebraic integer in
$\Q_{e/e'}\cap\Q_m=\Q_{\gcd(e/e',m)}$.  To prove the first statement,
it suffices to show that $e/e'$ and $m$ are relatively prime.
However, $e/e'=\dt(V)/re'$ and $\gcd(\dt(V),m)=re'$, so
$\gcd(\dt(V)/re',m)=1$ by assumption.  Any $m$ is large compared to a
square-free integer, so $\tnu_m\in \Z$ for all $m$ when $\dt(V)$ is
square-free.  Finally, divisors of square-free integers are
square-free, so if $\dt$ is square-free, so is $\dt(V)$.
\end{proof}

The untwisted version of this result is due to Kashina,
Sommerh\"{a}user, and Zhu~\cite{KashinaSommerZhu:2006}.

\section{The twisted exponents for $H_8$}
\label{tFSH8section}
%%%%%%%%%%%%%%%%%%%%%%

Let $H_8$ denote the Kac algebra of dimension $8$.  It is  
the smallest semisimple Hopf algebra which is neither commutative nor
cocommutative.  
 As an
algebra, $H_8$ is generated by elements $x$, $y$ and $z$,  with relations: 
\begin{equation*}
x^2=y^2=1, \; z^2= \dfrac{1}{2}\left(1+x+y-xy\right), \;xy=yx,\; xz=zy, \text{ and } yz=zx.
\end{equation*}
The coalgebra structure of $H_8$ is given by the following:
$$\Delta(x)=x \otimes x,\; \varepsilon(x)=1, \text{ and } S(x)=x,$$
$$\Delta(y)=y \otimes y, \; \varepsilon(y)=1, \text{ and } S(y)=y,$$
$$\Delta(z)=\dfrac{1}{2}\left(1 \otimes 1 + 1 \otimes x + y \otimes 1
- y \otimes x\right)\left(z \otimes z\right),$$ $$ \varepsilon(z)=1,
\text{ and } S(z)=z.$$ 
%The normalized integral is given by
%$$\Lambda=\dfrac{1}{8}\left(1+x+y+xy+z+xz+yz+xyz\right).$$ 
This Hopf algebra was first introduced by Kac and Paljutkin~\cite{KacPalj:1966}
and revisited later by Masuoka~\cite{Masuoka:1995}.

The Hopf algebra $H_8$ has $4$ one-dimensional representations and a
single two-dimensional simple module.  The characters for the
irreducible representations of $H_8$ are listed in
Table~\ref{tabH8irreps}, where $\chi_i$ corresponds to representation
$V_i$. The automorphism group of $H_8$ is the Klein four-group
(cf.~\cite{SageVega:2012}).  These automorphisms are given in
Table~\ref{autotableforH8}. All four automorphisms satisfy
$\tau^2=\Id$, with $\ord(\tau_1)=1$ and the others of order $2$.  The
twisted exponents of the irreducible representations are given in
Table~\ref{texpstable}; they were computed directly from the
definition using Mathematica.  The sixth column gives the twisted
exponents of the regular representation, which are easily computed
from the twisted exponents of the $V_i$'s and Corollary~ \ref{lcmV}.
\begin{center}
\begin{table}
%\textbf{Characters for the Irreducible Representations of $H_8$}\\
% use packages: array
\begin{tabular}{|c||c|c|c|c|c|c|c|c|}
\hline
 & $1$    & $x$ & $y$  & $xy$ & $z$ & $xz$ & $yz$ & $xyz$\\
\hline
\hline
$\chi_1$  & $1$  & $1$ & $1$  & $1$ & $1$ &1 & $1$& $1$\\
$\chi_2$  & $1$  & $1$ & $1$  & $1$ & $-1$ & $-1$& $-1$ & $-1$\\
$\chi_3$  & $1$  & $-1$ & $-1$  & $1$ & $i$ &$-i$ & $-i$ & $i$\\
$\chi_4$  & $1$  & $-1$ & $-1$  & $1$ & $-i$ &$i$ & $i$ &  $-i$ \\
$\chi_5$  & $2$  & $0$ & $0$  & $-2$ & $0$ & $0$ & $0$ & $0$\\
\hline
\end{tabular}
\vspace{3ex}
\caption{Characters for the Irreducible Representations of $H_8$}
\label{tabH8irreps}
 \end{table}
\end{center}
\begin{center}
\begin{table}
\begin{tabular}{|c||c|c|c|c|}
\hline
& $1$ & $x$& $y$& $z$\\
\hline
\hline
$\tau_1=\Id$& $1$ & $x$& $y$& $z$\\
$\tau_2$& $1$ & $x$& $y$ & $xyz$\\
$\tau_3$& $1$ & $y$& $x$& $\frac{1}{2}\left(z+xz+yz-xyz\right)$\\
$\tau_4$& $1$ & $y$& $x$& $\frac{1}{2}\left(-z+xz+yz+xyz\right)$\\
\hline
\end{tabular}
\vspace{3ex}
\caption{Automorphisms of $H_8$}
\label{autotableforH8}
\end{table}
\end{center}

%%%%%%%%%%%%%%

\begin{center}
\begin{table}
\begin{tabular}{|c||c|c|c|c|c|c|}
\hline
&$V_1$&$V_2$&$V_3$&$V_4$&$V_5$&$H_8$\\
\hline
\hline
$\exp_{\tau_1}(V_i)=\exp(V_i)$&$1$&$2$&$2$&$2$&$8$&$8$\\
$\exp_{\tau_2}(V_i)$&$1$&$1$&$1$&$1$&$4$&$4$\\
$\exp_{\tau_3}(V_i)$&$1$&$1$&$2$&$2$&$2$&$2$\\
$\exp_{\tau_4}(V_i)$&$1$&$1$&$2$&$2$&$2$&$2$\\
\hline
\end{tabular}
\vspace{3ex}
\caption{Twisted Exponents for  $H_8$}
\label{texpstable}
\end{table}
\end{center}

\bibliographystyle{amsalpha}
\bibliography{myrefs}                % expects file "myrefs.bib"

\end{document}